\documentclass[letterpaper, 10 pt, conference]{ieeeconf}

\IEEEoverridecommandlockouts                              
\usepackage{graphics} 
\usepackage{epsfig} 
\usepackage{mathptmx} 
\usepackage{times} 
\usepackage{amsmath} 
\usepackage{amssymb}  
\usepackage{balance}

\usepackage[mathcal]{euscript} 
\usepackage{mathtools}
\usepackage{mathrsfs}

\usepackage{graphicx}
\graphicspath{{Figures/}} 
\usepackage{caption} 
\usepackage{subfig} 
\usepackage{xcolor}

\newtheorem{thm}{Theorem}

\newtheorem{proposition}{Proposition}
\newtheorem{problem}[thm]{Problem}

\newcommand{\R}{{\mathbb R}}
\newcommand{\D}{{\mathbb D}}
\newcommand{\E}{{\mathbb E}}
\newcommand{\cN}{{\mathcal N}}

\newcommand{\cS}{{\mathcal S}}

\newcommand{\cQ}{{\mathcal Q}}

\newcommand{\tr}{\operatorname{trace}}

\newcommand{\argmin}{\operatornamewithlimits{argmin}} 
\newcommand{\ignore}[1]{}

\def\spacingset#1{\def\baselinestretch{#1}\small\normalsize}
\spacingset{.95}

\definecolor{grey}{rgb}{0.6,0.6,0.6}
\definecolor{lightgray}{rgb}{0.97,.99,0.99}

\definecolor{USred}{rgb}{0.74,0.1,0.1}
\definecolor{USblue}{rgb}{0.2,0.2,0.7}

\title{Optimal steering for non-Markovian Gaussian processes}
\author{Daniele Alpago, Yongxin Chen, Tryphon Georgiou and Michele Pavon
\thanks{D.~Alpago is with the Dipartimento di Ingegneria dell'Informazione, Universit\`a di Padova, 35131 Padova, Italy; {daniele.alpago@phd.unipd.it}}
\thanks{Y.~Chen is with the School of Aerospace Engineering, Georgia Institute of Technology, Atlanta, GA 30332;{yongchen@gatech.edu}}
\thanks{T.T.\ Georgiou is with the Department of Mechanical and Aerospace Engineering,
University of California, Irvine, CA 92697; {tryphon@uci.edu}}
\thanks{M.\ Pavon is with the Dipartimento di Matematica ``Tullio Levi-Civita",
Universit\`a di Padova, 35121 Padova, Italy; {pavon@math.unipd.it}}
\thanks{
Supported in part by the
NSF under grants 1509387, 1901599,
the AFOSR under grants FA9550-15-1-0045 and FA9550-17-1-0435, and by the University of Padova Research Project CPDA 140897.}}

\begin{document}

\maketitle
\begin{abstract}
At present, the problem to steer a non-Markovian process with minimum energy between specified end-point marginal distributions remains unsolved. Herein, we consider the special case for a 
non-Markovian process $y(t)$ which, however, assumes a finite-dimensional stochastic realization with a Markov {\em state process} that is fully observable. In this setting, and over a finite time horizon $[0,T]$, we determine an optimal (least) finite-energy control law that steers the stochastic system to a final distribution that is compatible with a specified distribution for the terminal output process $y(T)$; the solution is given in closed-form. This work provides a key step towards the important problem to steer a stochastic system based on partial observations of the state (i.e., an output process) corrupted by noise, which will be the subject of forthcoming work.
\end{abstract}
\section{Introduction}
Throughout we will be considering a controlled evolution of the vector Gauss-Markov process $\{x(t) \mid 0\le t\le T\}$ that obeys the linear stochastic differential equation
\begin{subequations}
\begin{align}\label{controlled}&dx^u=A(t)x^u(t)dt+B(t)u(t)+B(t)dw(t),\\\nonumber &x^u(0)=\xi \mbox{ a.s.}
\end{align}
Here, as it is standard, $w$ is an $m$-dimensional Wiener process and $\xi$ is an $n$-dimensional Gaussian random vector which is independent of $w$. For simplicity we suppose that $\xi$ has zero mean, and that it has density
\begin{equation}\label{initial}\rho_0(x)=(2\pi)^{-n/2}\det (\Sigma^x_0)^{-1/2}\exp\left\{-\frac{1}{2}x'\left(\Sigma^x_0\right)^{-1}x\right\}.
\end{equation}
As it is common, we also assume that $A(\cdot)$ and $B(\cdot)$ are continuous matrix functions taking values in $\R^{n\times n}$ and $\R^{n\times m}$, respectively.

For this setting, in recent years, there has been considerable interest in the problem of {\em minimum-energy steering} of the (Gaussian) distribution of $x(t)$ to a target distribution $\mathcal N(0,\Sigma_T^x)$ at time $t=T$, \cite{CGP1,CGP3,HW,GT,bakolas}. Important extensions include \cite{CGP3} the more challenging case when the control process $u$ and the noise $w$ enter through different channels (i.e., having different ``input matrices'' $B$ in \eqref{controlled}), and the infinite-horizon case where the goal is to achieve with minimum power a specified stationary state \cite{CGP3}; the latter generalizes the classical work on covariance control of Skelton {et al.} \cite{HS,GS}. Motivation for such problems is manifold: they represents a most natural relaxation of classical LQR steering problems and have important applications in quality control and industrial manufacturing, vehicle path planning \cite{OT}, statistical physics as in {\em cooling} and control of nano-to-meter scale resonators, atomic force microscopy and so forth, see e.g., \cite{FilHonStr08,CGPcooling}.   

Historically, the origin of the steering problem stems from a {\em Gedankenexperiment} formulated by Schr\"odinger in the early thirties \cite {S1,S2}, seeking the most likely flow of particle distributions between observed end-point marginals. Schr\"odinger's problem amounted to a problem in the theory of large deviations (which was unavailable at that time). Indeed, thanks to Sanov's theorem \cite{SANOV},
the Schr\"odinger's problem amounts to seeking a probability distribution on particle trajectories having maximum entropy andwhich  is in agreement with the end-point specified marginal distributions \cite{For,Beu,Jam2,F2,Wak}. Then, in the late eighties and early nineties, following work of  Jamison, F\"ollmer, Nagasawa, Wakolbinger, Fleming, Holland, Mitter and others, a clear connection was made with stochastic control \cite{DP,DPP,PavWak91}.  The distribution on paths, corresponding to the  uncontrolled evolution, plays the role of the ``prior" measure in the maximum entropy problem which generalizes Schr\"odinger's original one. At about the same time, Blaquiere  \cite{blaquiere1992controllability}  studied the control of the Fokker-Planck equation and later Brockett studied the Louiville equation \cite{brockett2012notes} along a similar spirit, to steer distributions to a target one. This circle of control problems for uncertain system has recently been linked to yet another fast developing topic, Optimal Mass Transport (OMT) problem \cite{Vil}, when it was realized that {\em Schr\"odinger's  bridge problem} (SBP, as it seeks to ``bridge'' the two end-point marginals) may be viewed as a regularization of OMT and provides an effective computational approach to the latter \cite{Mik, mt, MT,leo,CGP5}.

Extending the Schr\"odinger problem to the case of non-Markov processes is a tantalizing one and a natural next step. While the general case is currently wide open, in the present paper we work out the special of steering the output of a Gauss-Markov model. More specifically, in conjunction with \eqref{controlled}, we consider the output process
\begin{equation}\label{output}
y(t)=C(t) x^u(t),
\end{equation}
\end{subequations}
where $C(\cdot)$ is continuous and takes values in $\R^{p\times n}$ for $p<n$. For instance, this case 
arises when we consider steering only some components of the state to a prescribed terminal distribution (see \ref{examples}). Clearly, $y$ by itself is not a Markov process. Thus, this seemingly innocuous problem falls into the category of Schr\"odinger bridge problems with non-Markov prior for which the form of the optimal control is, in general, unknown\footnote{See \cite{P94} for a considerably simpler ``half-bridge" problem where only the final distribution is prescribed.}. Problems where only a portion of the state needs to be specified arise, for instance, in thickness control (film extrusion) \cite{astrom,bakolas} where the remaining components of the state vector might either not be of interest or may be difficult/expensive to measure. In Section \ref{examples} we discuss a case where it is of interest to regulate only the distribution of the momentum of stochastic oscillators. 


The outline of the paper is as follows. In Section \ref{background}, we recall some central results from \cite{CGP1} in the case of a Markovian prior. In Section \ref{formulation}, we give a precise formulation of our stochastic control problem. In Section \ref{sec:finitehorizon}, we provide a closed-form solution to our problem by finding the terminal time state covariance which can be reached with minimum energy among those complying with the assigned covariance of $y(T)$. Finally, Section \ref{examples} illustrates the results in a problem of steering the momentum distribution of a stochastic oscillator to a desired one.

\section{Background}\label{background}
Let $\mathcal U(\Sigma^x_0,\Sigma^x_T)$ be the family of {\em adapted\footnote{$u(t)$ only depends on $t$ and on $\{x^u(s); 0\le s\le t\}$ for each $t\in [0,T]$.}, finite energy} control functions such that (\ref{controlled}) has a strong solution on $[0,T]$ and $x(T)$ has distribution  $\mathcal N(0,\Sigma_T^x)$. The optimal steering problem reads
\begin{problem}\label{pro:steering} 
Determine 
\[
u^*:= \argmin_{u\in \mathcal U(\Sigma^x_0,\Sigma^y_T)} \,J(u):=\E\left\{\int_0^Tu(t)' u(t) \,dt\right\}.
\]
\end{problem}
In \cite[Theorem 8]{CGP1}, it was shown that, under controllability of the pair $(A(\cdot),B(\cdot))$ on the given time interval, $\mathcal U(\Sigma^x_0,\Sigma^x_T)$ is nonempty and the (unique) optimal control is a linear feedback of the state given by
\begin{equation}\label{optcontr}
u^\star(t)=-B(t)'Q(t)^{-1}x(t),
\end{equation}
where $P(t)$ and $Q(t)$, taking values in the set of symmetric, $n\times n$ matrices, are the unique {\em nonsingular} solutions on $[0,T]$ of the system of linear matrix equations
\begin{subequations}\label{system_Lyap}
\begin{align}\label{system_Lyap_P}
\dot{P}(t)&=A(t)P(t)+P(t)A(t)'+B(t)B(t)',\\ \label{system_Lyap_Q}
\dot{Q}(t)&=A(t)Q(t)+Q(t)A(t)'-B(t)B(t)',
\end{align}
\end{subequations}
nonlinearly coupled through the boundary conditions
\begin{subequations}\label{boundary_Lyap}
\begin{align}
(\Sigma^x_0)^{-1}&=P(0)^{-1}+ Q(0)^{-1},\\
(\Sigma^x_T)^{-1}&=P(T)^{-1}+Q(T)^{-1}.
\end{align}
\end{subequations}
The solutions to these equations can actually be provided in closed form as a function of $(\Sigma^x_0,\Sigma^x_T)$,
see \cite[Section III]{CGP1} for further details.

Let $P_0$ and $P_u$ be the probability measures on $C(0,T;\R^n)$, the $n$-dimensional continuous functions corresponding to the solutions of (\ref{controlled}) with control $0$, and $u\in\mathcal U$, respectively. Also let $\pi_0(x_0,x_T)$ and $\pi_u(x_0,x_T)$ be their initial-final joint density, respectively.  In \cite[Section IV]{CGP1}, a well known decomposition of the relative entropy \cite{F2} was extended to the case of degenerate diffusions, to show that the Schr\"odinger bridge problem with marginals densities $\rho_0=\mathcal N(0,\Sigma_0^x)$ and $\rho_T=\mathcal N(0,\Sigma_T^x)$  can be reduced to the following maximum entropy problem for distributions on a finite-dimensional space:
\begin{problem}\label{static}
	Minimize over densities $\pi_u$ on $\R^n\times\R^n$ the Kullback-Leibler index
	\begin{equation}\label{staticindex}
	\D(\pi_u\|\pi_0):=\int\int\left[\log\frac{\pi_u(x,y)}{\pi_0(x,y)}\right]\pi_u(x,y)dxdy
	\end{equation}
	subject to the (linear) constraints
	\begin{equation}\label{constraint} \int \pi_u(x,y)dy=\rho_0(x),\quad \int \pi_u(x,y)dx=\rho_T(y).
	\end{equation}
\end{problem}

Let $\Sigma^u_{0,T}$ be the covariance of $\pi_u(x_0,x_T)$. Since $u\in\mathcal U(\Sigma^x_0,\Sigma^x_T)$, $\Sigma^u_{0,T}$ has necessarily the structure
\begin{equation}\label{structure}\Sigma^u_{0,T}=\left[\begin{array}{cc}\Sigma^x_0 &Y^u \\ (Y^u)' & \Sigma^x_T\end{array}\right]
\end{equation}
for some $Y^u$.
Let $S_{0,T}$ instead be the covariance corresponding to $\pi_0(x_0,x_T)$. Then, it has the form
\begin{equation}\label{eq:S}
S_{0,T}=\left[\begin{matrix}\Sigma^x_0&\Sigma^x_0\Phi(T,0)'\\
\Phi(T,0)\Sigma_0^x&S_T\end{matrix}\right]
\end{equation}
where
\[
S_T=\Phi(T,0)\Sigma^x_0\Phi(T,0)'+\int_0^T\Phi(T,\tau)B(\tau)B(\tau)'\Phi(T,\tau)'d\tau,
\]
with $\Phi(t,s)$ denoting the state-transition of $A(\cdot)$ determined by
\[
\frac{\partial}{\partial t}\Phi(t,s)=A(t)\Phi(t,s), \quad \Phi(t,t)=I.\]
Thanks to the explicit form of relative entropy (Kullback-Leibler index) for Gaussian distributions \cite{COVER_THOMAS}, Problem \ref{static} can be expressed in terms of covariances as follows:
\begin{equation}\label{eq:remx}
\argmin_{(Y^u)\in\cQ^x} \quad -\log\det\Sigma^u_{0,T} + \tr(S_{0,T}^{-1}\,\Sigma^u_{0,T})
\end{equation} 
where $\Sigma^u_{0,T}$ is as in (\ref{structure}) and
\[
\cQ^x:=\left\{Y\in\R^{n\times n}:\,\Sigma^x_T-Y'(\Sigma^x_0)^{-1}Y>0 \right\},
\]
see \cite[Section IV]{CGP1} for the details.

\section{Problem formulation}\label{formulation}
We consider the output process in \eqref{output} and assume that the state $x(t)$ is fully observable.  The finite-dimensional {\em Markovian representation} ({\em stochastic realization}) for $y$ provided by (\ref{controlled})-(\ref{output}) is available. Such a representation, as is well-known, constitutes the starting point of Kalman filtering and much of optimal control theory, and the construction of such a model with minimal state vector dimension has been the subject of intense study \cite{LP}. This too is our starting point.


Let us denote by $\mathcal U(\Sigma^x_0,\Sigma^y_T)$ be the family of adapted control functions such that (\ref{controlled}) has a strong solution on $[0,T]$ and $y(T)$ has distribution  $\mathcal N(0,\Sigma_T^y)$.  We formulate the following {\em Schr\"{o}dinger Bridge Problem} with non-Markov prior:

\begin{problem}\label{formalization} Determine 
\[
u^*:= \argmin_{u\in \mathcal U(\Sigma^x_0,\Sigma^y_T)} \,J(u):=\E\left\{\int_0^Tu(t)' u(t) \,dt\right\}.
\]
\end{problem}
Notice that on one side, at $t=0$, the boundary constraint requires matching the covariance for the state vector (which can be relaxed) while on the other end, at $t=T$, requires matching the covariance of the output
\begin{equation}\label{terminalcondition}\Sigma^y_T=C(T)\Sigma^x_TC(T)'.
\end{equation} 
The value of $\Sigma_T^x$ is a parameter and there are in general several values for it such that (\ref{terminalcondition}) is satisfied\footnote{The case where only $\Sigma^y_0$ and $\Sigma^y_T$ are prescribed can be treated in a similar fashion by optimizing also with respect to $\Sigma_0^x$.}. Corresponding to each one of them, there is a feedback control in $\mathcal U(\Sigma^x_0,\Sigma^x_T)$ optimally performing the transfer of distributions according to \cite{CGP1}. Thus, the problem may be also viewed as that of determining the one final covariance $\Sigma_T^x$, among those  compatible with $\Sigma_T^y$, whose corresponding optimal control (\ref{optcontr}) has minimum energy.

Inspired by the reduction of the classical case leading to Problem \ref{static}, we proceed in the next section to derive a closed-form solution of Problem \ref{formalization}.

\section{Solution to the non-Markovian steering problem}\label{sec:finitehorizon}
In view of \eqref{eq:remx} in Section \ref{background}, Problem  \ref{formalization} can be rewritten as
\begin{equation}\label{eq:soc}
\begin{aligned} 
\argmin_{u\in\mathcal{U}} &\quad \E\left\{\int_0^T\,u(t)'u(t)\,dt\right\}\\
\text{subject to } &\quad x(0)\sim\cN(0,\Sigma^x_0),\quad x(T)\sim\cN(0,X),\\
&\quad CXC'=\Sigma_T^y,
\end{aligned}
\end{equation}
where $\Sigma^x_0$, $\Sigma_T^y$ constitute the given data while $X$ is a parameter. This can be further recast as
\begin{problem}\label{relentropyformul} Given $\Sigma^x_0$, $\Sigma_T^y$, and $S=S_{0,T}$ as in (\ref{eq:S}), determine
\begin{equation}\label{eq:rem}
\begin{aligned} 
\argmin_{(X,Y)\in\cQ} &\quad -\log\det\Sigma + \tr(S^{-1}\,\Sigma)
\end{aligned}
\end{equation}
subject to $\Sigma=\begin{bmatrix}\Sigma^x_0 & Y\\Y' & X\end{bmatrix}>0$ and $CXC'=\Sigma_T^y$.
\end{problem}

Now, let
\[
S^{-1}=
\begin{bmatrix}
N & V\\
V' & P
\end{bmatrix},
\]
and
\[
\cQ:=\left\{(X,Y)\in\cS_+\times\R^{n\times n}:\,X-Y'(\Sigma^x_0)^{-1}Y>0 \right\}
\]
where $\cS_+$ is the set of $n\times n$ symmetric positive definite matrices. 

We construct below the Lagrangian $\mathcal{L}$ introducing a Lagrange multiplier and consider the unconstrained minimization
\begin{equation}\label{eq:inflag}
	\inf_{(X,Y)\in\cQ} \mathcal{L}(X,Y,M).
\end{equation}
The Lagrangian is given by (we write, for simplicity, $\Sigma_0$ instead of $\Sigma_0^x$)
\begin{align*}
   \mathcal{L}(X,Y,M) 
   &=-\log\det\left(X-Y'\Sigma_0^{-1}Y\right)+\tr(PX)\\&+\tr\left[M(CXC' - \Sigma_T^y)\right]+2\tr(V' Y)+c,	
\end{align*}
where $M=M'$ is a Lagrange multiplier and $c\in\R$ is a constant term. We first  check the convexity of $\mathcal{L}$ with respect to $(X,Y)$. 
\begin{proposition}  $\mathcal{L}$ is jointly convex in $(X,Y)$ over $\cQ$.
\end{proposition}
\begin{proof}
Let $\delta\mathcal{L}:=\delta\mathcal{L}(X,Y,M;\delta X, \delta Y)$ denoting the first variation of $\mathcal{L}$ in the direction $(\delta X, \delta Y)$. Applying the chain rule, 
\begin{align*}
\delta\mathcal{L}
=&-\tr[(X-Y'\Sigma_0^{-1}Y)^{-1}\,\delta\left(X-Y'\Sigma_0^{-1}Y;\delta X, \delta Y\right)]\\
&+\tr\left[(P+C' M C)\delta X + 2V'\delta Y\right]\\
=&-\tr[(X-Y'\Sigma_0^{-1}Y)^{-1}(\delta X-Y'\Sigma_0^{-1}\delta Y - \delta Y'\Sigma_0^{-1}Y)]\\
&+\tr\left[(P+C' M C)\delta X+2V' \delta Y\right].
\end{align*}
To check the convexity it is sufficient look at the diagonal of the ``Hessian" of $\mathcal{L}$
\[
\delta^2\mathcal{L}:=\delta\mathcal{L}(X,Y,M;\delta X, \delta X, \delta Y, \delta Y).
\]
We have
\begin{align*}
\delta^2\mathcal{L} 
=&\tr\left[ \left((X-Y'\Sigma_0^{-1}Y)^{-1}(\delta X-Y'\Sigma_0^{-1}\delta Y - \delta Y'\Sigma_0^{-1}Y)\right)^2\right]\\
&+ 2\tr\left[(X-Y'\Sigma_0^{-1}Y)^{-1}(\delta Y'\Sigma_0^{-1}\delta Y)\right].
\end{align*}
which is clearly non-negative on $\cQ$.
\end{proof}

To find the minimum of $\mathcal{L}$ in $\cQ$ is therefore sufficient to solve
\begin{equation*}
   \delta\mathcal{L}(X,Y,M;\delta X, \delta Y) = 0,\,\qquad \forall\,(\delta X,\delta Y)\in\cS\times\R^{n\times n},
\end{equation*}
from which we get the two equations
\begin{align}
&P+C' MC-(X-Y'\Sigma_0^{-1}Y)^{-1}=0, \label{eq:varX}\\
&V+\Sigma_0^{-1}Y(X-Y'\Sigma_0^{-1}Y)^{-1}=0. \label{eq:varY}
\end{align}
To compute the optimal $(X,Y)$, we use these equations in the Lagrangian and then proceed to maximize the resulting (concave) functional with respect to $M$. Accordingly, the last equation we need is given by
\begin{equation}
\delta\mathcal{L}(X,Y,M;\delta M) = 0,\;\forall\,\delta M\in\cS
\;\iff\;  CXC'=\Sigma_T^y. \label{eq:varM}
\end{equation}

Let $Z:=X-Y'\Sigma_0^{-1}Y$ and note that $Z=Z'>0$. We immediately get $X=Z+Y'\Sigma_0^{-1}Y$ and
\begin{align}
	&\eqref{eq:varX} \quad\iff\quad Z^{-1} = P+C' M C, \label{eq:forZ1}\\
	&\eqref{eq:varY} \quad\iff\quad Y = -\Sigma_0VZ \notag.
\end{align}
Therefore, $X=Z+ZV'\Sigma_0VZ$ and 
\begin{equation}\label{eq:forZ2}
	\eqref{eq:varM} \quad\iff\quad CZC' + CZV'\Sigma_0VZC' =\Sigma_T^y.
\end{equation}
At this point we only need to find $Z$ from equations \eqref{eq:forZ1}, \eqref{eq:forZ2}. Since we can always find a state space transformation $\mathcal{T}$ such that $C\,\mathcal{T}=[I\,|\,0]$ (or a change of basis in the outputs' space), without loss of generality, we can always assume that $C=[I\,|\,0]$. Let
\begin{equation*}
	Z=\begin{bmatrix}
	Z_{11} & Z_{12}\\
	Z_{21} & Z_{22}
	\end{bmatrix},
	\quad
	V=\begin{bmatrix}
	V_{11} & V_{12}\\
	V_{21} & V_{22}
	\end{bmatrix},
	\quad
	P=\begin{bmatrix}
	P_{11} & P_{12}\\
	P_{21} & P_{22}
	\end{bmatrix}.
\end{equation*}
Equation \eqref{eq:forZ2} becomes 
\begin{equation}\label{eq:forZ11}
Z_{11}+\begin{bmatrix}Z_{11} & Z_{12}\end{bmatrix}
\underbrace{\begin{bmatrix}
K_{11} & K_{12}\\
K_{21} & K_{22}
\end{bmatrix}}_{V'\Sigma_0V\,>\,0}
\begin{bmatrix}Z_{11} \\ Z_{12}\end{bmatrix} = \Sigma_T^y,
\end{equation}
while equation \eqref{eq:forZ1} can be equivalently written as 
\begin{equation}\label{eq:forM}
\begin{bmatrix}
I & 0\\
0 & I
\end{bmatrix}
=
\begin{bmatrix}
Z_{11}(M+P_{11})+Z_{12}P_{12} & Z_{11}P_{12}+Z_{12}P_{22}\\
Z_{21}(M+P_{11})+Z_{22}P_{21} & Z_{21}P_{12}+Z_{22}P_{22}
\end{bmatrix}
\end{equation}
which reduces to the system of equations
\begin{equation}\label{eq:sysZ}
\left\{
\begin{split}
&Z_{11}P_{12}+Z_{12}P_{22}=0\\
&Z_{21}P_{12}+Z_{22}P_{22}=I\\
&Z_{12} = Z_{21}'
\end{split}
\right.
\iff
\left\{
\begin{split}
&Z_{21}=-P_{22}^{-1}P_{21}Z_{11}\\
&Z_{12} = Z_{21}'\\
&Z_{22}=P_{22}^{-1}-Z_{21}P_{12}P_{22}^{-1}
\end{split}
\right.
\end{equation}
Plugging $Z_{12}$, $Z_{21}$ and $Z_{22}$ into \eqref{eq:forZ11}, we get
\begin{equation}\label{eq:quadZ11}
Z_{11}+Z_{11}\,A\,Z_{11} = \Sigma_T^y,
\end{equation}
where
\[
A:=\begin{bmatrix}I & -P_{12}P_{22}^{-1}\end{bmatrix}
\begin{bmatrix}
K_{11} & K_{12}\\
K_{21} & K_{22}
\end{bmatrix}
\begin{bmatrix}I \\ -P_{22}P_{21}\end{bmatrix}>0.
\]
Equation \eqref{eq:quadZ11} is a quadratic equation with two solutions
\begin{equation}\label{eq:Z}
Z_{11}^{\pm} = A^{-\frac{1}{2}}\left[\pm\left(\frac{1}{4}I+A^\frac{1}{2}\Sigma_T^yA^\frac{1}{2}\right)^\frac{1}{2}-\frac{1}{2}I\right]A^{-\frac{1}{2}}.
\end{equation}
Clearly,  $Z=X-Y'\Sigma_0^{-1}Y>0$ by Schur complement, which implies $Z_{11}>0$. This singles out the solution $Z_{11}^+$.
We can now recover $Z$ from \eqref{eq:sysZ} and then  $X=Z+ZV'\Sigma_0VZ$ and $Y = -\Sigma_0VZ$. Finally, from \eqref{eq:forM}, one can find the multiplier $M$: 
\[
   M = (Z^+_{11})^{-1}-P_{11}-(Z_{11}^+)^{-1}Z_{12}P_{12}.
\]
The above results can be summarized as follows.

\begin{thm}
	Let $Z_{11}^{+}$ be as in \eqref{eq:Z} and $Z, X, Y$ be derived accordingly, then $(X, Y)$ solves Problem \ref{relentropyformul}. Furthermore, the solution to Problem \ref{formalization} coincides with the solution to Problem \ref{pro:steering} with $\Sigma_T^x = X$.
\end{thm}

\section{Example}\label{examples}
Consider controlling the Ornstein-Uhlenbeck model of physical Brownian motion
\begin{equation}\label{eq:BMexu}
\begin{split}
	dq^u(t) &= p^u(t)\\
	dp^u(t) &= -\beta\,p^u(t)dt - Kdt + u(t)dt + dw(t)
\end{split}
\end{equation}
corresponding to a given quadratic potential $V(q)=\frac{1}{2}q'Kq$ with $K$ symmetric, positive-definite, and $u(\cdot)$ is the control force. By setting
\[
   x=\begin{pmatrix}
   q\\p
   \end{pmatrix},\quad
   A=\begin{pmatrix}
   0 & I\\
   -K & -\beta I
   \end{pmatrix},\quad
   B=\begin{pmatrix}
   0\\
   I
   \end{pmatrix},   
\]
model \eqref{eq:BMexu} becomes
\begin{align*}
  &dx^u(t) = A\,x^u(t) + B\,u(t) + B\,dw(t)\\
  &x^u(0) = \xi \text{ a.s.}
\end{align*}
where $\xi$ is zero-mean Gaussian
with $\Sigma_0^x=I/2$, and the pair $(A,B)$ is controllable. We consider a state dimension of $n=2$ and we assume for simplicity that the units are such that $K=I$ and $\beta=1$. 

We would like to steer the Gaussian distribution of the momentum equal to a final distribution at time $T=1$ with $\Sigma_1^p=1/16$ minimizing the quadratic control energy
under the controlled dynamics \eqref{eq:BMexu}. In other words, we are prescribing only the final covariance matrix of $y(t)=C\,x(t)$ with $C=\left[0\,|\,I\right]$. Figure \ref{fig:phctrvv} shows the trajectories of the state variables in the phase space (left) and the corresponding control efforts (right), i.e. the intersections of the phase plot with the slice planes $p$ and $q$ respectively. 
\begin{figure}[h!]
	\centering
	\includegraphics[scale=0.57]{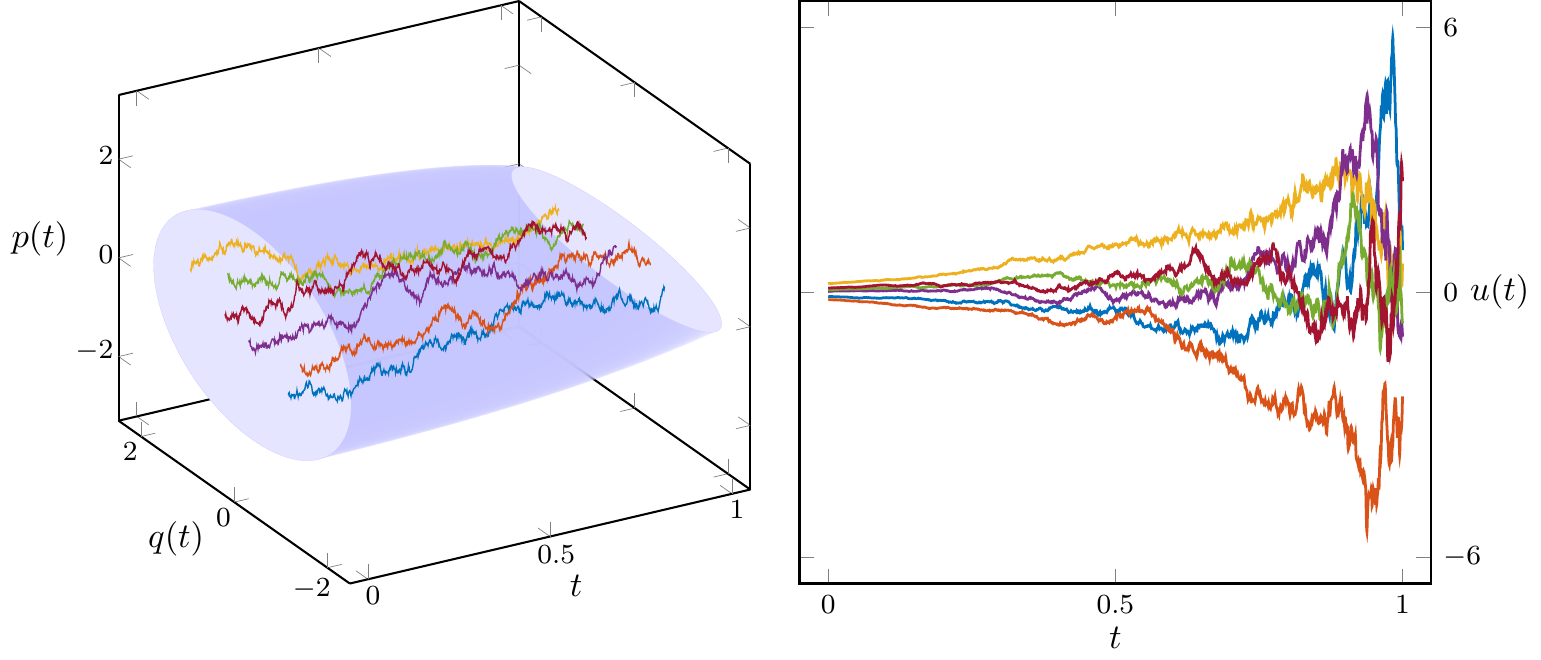}
	\caption{Realizations on the phase space (left) and relative control efforts (right). Control on the momentum.}
	\label{fig:phctrvv}
\end{figure}
Figure \ref{fig:xvsecvv} highlights instead the trajectories of position (left) and momentum (right) with the corresponding confidence interval.
\begin{figure}[h!]
	\centering
	\includegraphics[scale=0.58]{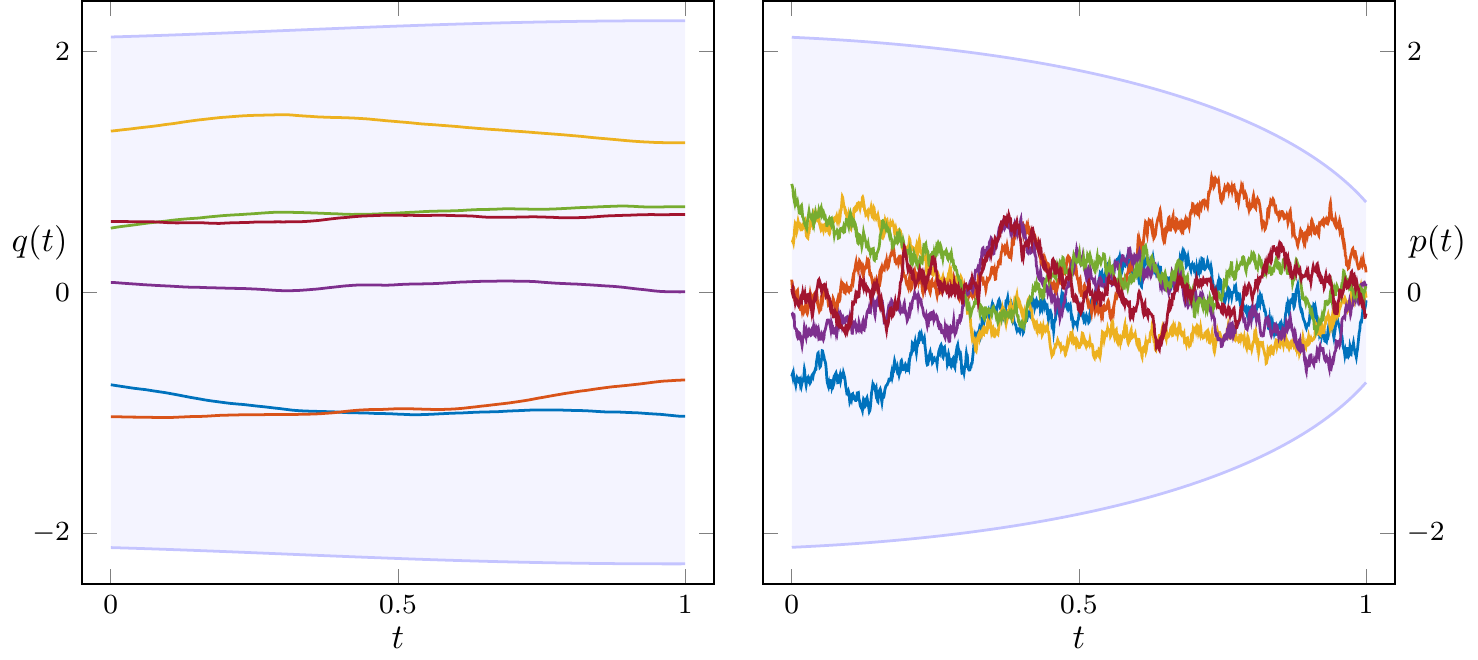}
	\caption{Position's trajectories (left) and momentum's trajectories (right). Control on the momentum.}
	\label{fig:xvsecvv}
\end{figure}
In all the figures, the transparent blue tube represent the "$3\sigma$" confidence interval, i.e. its intersection with the slice plane $t$ is given by
\[
\left\{ (q,p)\in\R^2\,\bigg|\,
\begin{bmatrix}
q & p
\end{bmatrix}
\Sigma_t^{-1}
\begin{bmatrix}
q\\p
\end{bmatrix}
\le 3^2
\right\}.
\]%
The figures highlight the reduction of the variance of the momentum process as time increases to $T=1$.


\begin{thebibliography}{99}
\bibitem{astrom} K. J. \AA str\"om, {\em Introduction to Stochastic Control Theorey}, Academic Press, 1970.
%
\bibitem{bakolas}
Efstathios Bakolas, Finite-horizon covariance control for discrete-time stochastic linear systems subject to input constraints, {\em Automatica}, {\bf 91}, pp. 61-68, 2018. 
\bibitem{Beu} A. Beurling, An automorphism of product measures, {\em Ann.
Math.} {\bf 72} (1960), 189-200.
%

\bibitem{blaquiere1992controllability}
A.~Blaqui\`ere, ``Controllability of a {F}okker-{P}lanck equation, the
  {S}chr{\"o}dinger system, and a related stochastic optimal control (revised
  version),'' \emph{Dynamics and Control}, vol.~2, no.~3, pp. 235--253, 1992.




\bibitem{brockett2012notes}
R.~Brockett, ``Notes on the control of the {L}iouville equation,'' in
  \emph{Control of Partial Differential Equations}.\hskip 1em plus 0.5em minus
  0.4em\relax Springer, 2012, pp. 101--129.

%
%
%
%
%
%
\bibitem{CGP1} Y.\ Chen, T.T.\ Georgiou and M.\ Pavon, ``Optimal steering of a linear stochastic system to a final probability distribution, Part I'',  {\em IEEE Trans. Aut. Control}, {\bf 61}, Issue 5, 1158-1169, 2016.
\bibitem{CGP3}Y.\ Chen, T.T.\ Georgiou and M.\ Pavon, ``Optimal steering of a linear stochastic system
to a final probability distribution, Part II'',  {\em IEEE Trans. Aut. Control}, {\bf 61}, Issue 5, 1170-1180, 2016.
\bibitem{CGPcooling}Y.\ Chen, T.T.\ Georgiou and M.\ Pavon, ``Fast cooling for a system of stochastic oscillators'', {\em J. Math. Phys.}, {\bf 56}, n.11, 113302, 2015.
\bibitem{CGP2016} Y.\ Chen, T.T.\ Georgiou and M.\ Pavon, On the relation between optimal transport and Schr\"{o}dinger bridges: A stochastic control viewpoint, {\em J. Optim. Theory and Applic.}, {\bf 169} (2), 671-691, 2016.
\bibitem{CGP5} Y.\ Chen, T.T.\ Georgiou and M.\ Pavon, Optimal transport over a linear dynamical system, {\em IEEE Trans. Aut. Control}, {\bf 62}, n. 5, 2137-2152, 2017.
%
%
%
\bibitem{COVER_THOMAS}
T.~M. Cover and J.~A. Thomas, {\em {E}lements of {I}nformation {T}heory}, Wiley, New York, 1991.

 \bibitem{DP} P. Dai Pra, A stochastic control approach to reciprocal diffusion processes, {\em Appl. Math. and Optimiz.}, {\bf 23} (1), 1991, 313-329.


\bibitem{DPP} P.Dai Pra and M.Pavon, On the Markov processes of Schroedinger, the Feynman-Kac formula and stochastic control, in {\em Realization and Modeling in System Theory} - Proc. 1989 MTNS Conf., M.A.Kaashoek, J.H. van Schuppen, A.C.M. Ran Eds., Birk\"auser, Boston, 1990, 497- 504.
%
%

\bibitem{FilHonStr08}
R.~Filliger, M.-O. Hongler, and L.~Streit, ``Connection between an exactly
  solvable stochastic optimal control problem and a nonlinear
  reaction-diffusion equation,'' \emph{Journal of Optimization Theory and
  Applications}, vol. 137, no.~3, pp. 497--505, 2008.


\bibitem{F2} H. F\"{o}llmer, Random fields and diffusion processes, in:
{\em \`Ecole
d'\`Et\`e de Probabilit\`es de Saint-Flour XV-XVII}, edited by P. L.
Hennequin, Lecture Notes in Mathematics, Springer-Verlag, New York, 1988, vol.1362,102-203.

\bibitem{For} R. Fortet, R\'esolution d'un syst\`eme d'equations de M.
Schr\"{o}dinger, {\em J. Math. Pure Appl.} IX (1940), 83-105.
%
\bibitem{GT} M. Goldshtein and P. Tsiotras, Finite-horizon covariance control of linear time-varying systems, in {\em IEEE Conference on Decision andControl}, Melbourne, Australia, Dec. 12-15, 2017, pp. 3606-3611.

\bibitem{GS} K. M. Grigoriadis and R. E. Skelton, Minimum-energy covariance controllers, {\em Automatica}, {\bf 33}, no. 4, pp. 569-578, 1997.

\bibitem{HW} A. Halder and E. D. Wendel, Finite horizon linear quadratic Gaussian density regulator with Wasserstein terminal cost, in {\em American Control Conference (ACC)}, 2016. IEEE, 2016, pp. 7249-7254.

\bibitem{HS} A. Hotz and R. E. Skelton, Covariance control theory, {\em International Journal of Control}, {\bf 46}, (1):13-32, 1987.
%
%
\bibitem{Jam2} B. Jamison, The Markov processes of Schr\"{o}dinger, {\em Z.Wahrscheinlichkeitstheorie verw. Gebiete} {\bf 32} (1975), 323-331.

%
%
%
%
%
\bibitem{leo} C. L\'eonard, A survey of the Schroedinger problem and some of its connections with optimal transport, {\em Discrete Contin. Dyn. Syst. A}, 2014, {\bf 34} (4): 1533-1574.
%
\bibitem{LP} A. Lindquist and G. Picci, {\em Linear Stochastic Systems: A Geometric Approach to Modeling, Estimation and Identification}, Springer, 2015.
%
\bibitem{Mik} T. Mikami, Monge's problem with a quadratic cost by the zero-noise limit of h-path processes, {\em Probab. Theory Relat. Fields}, {\bf 129}, (2004), 245-260.
%
\bibitem{mt} T. Mikami and M. Thieullen, Duality theorem for the stochastic optimal control problem., {\em Stoch. Proc. Appl.}, 116, 1815-1835 (2006).
%
\bibitem{MT} T. Mikami and M. Thieullen, Optimal Transportation Problem by Stochastic Optimal Control, {\em SIAM Journal of Control and Optimization}, {\bf 47}, N. 3, 1127-1139 (2008).
%
%
%
%
%
%
\bibitem{OT} K. Okamoto and P. Tsiotras, Optimal Stochastic Vehicle Path Planning Using Covariance Steering, {\em International Conference on Robotics and Automation}, Montreal, Canada, May. 20?24, 2019.

\bibitem{PavWak91}
M.~Pavon and A.~Wakolbinger, ``On free energy, stochastic control, and
  {S}chr{\"o}dinger processes,'' in \emph{Modeling, Estimation and Control of
  Systems with Uncertainty}.\hskip 1em plus 0.5em minus 0.4em\relax Springer,
  1991, pp. 334--348.



\bibitem {P94}M.Pavon, Stochastic control and non-Markovian Schr\"{o}dinger processes,in {\em Systems and Networks: Mathematical Theory and Applications}, vol. II, U. Helmke, R.Mennichen and J.Saurer Eds., Mathematical Research vol.79, Akademie Verlag, Berlin, 1994, 409-412.

\bibitem{SANOV} I. S. Sanov, On the probability of large deviations of random magnitudes (in Russian), {\em Mat. Sb. N. S.}, {\bf 42} (84) (1957). Select. Transl. Math. Statist. Probab., 1, 213-244 (1961).
%
\bibitem{S1} E. Schr\"{o}dinger, \"{U}ber die Umkehrung der Naturgesetze,
{\em Sitzungsberichte
der Preuss Akad. Wissen. Berlin, Phys. Math. Klasse} (1931), 144-153.

\bibitem{S2} E. Schr\"{o}dinger, Sur la th\'{e}orie relativiste de
l'\'{e}lectron et l'interpretation de
la m\'{e}canique quantique, {\em Ann. Inst. H. Poincar\'{e}} {\bf 2},
269 (1932).
%
\bibitem{Sin64} R. Sinkhorn, A relationship between arbitrary positive matrices and doubly stochastic matrices,  {\em Ann. Math. Statist.}, {\bf 35} (1964), 876-879.
%
%
%
%
%
\bibitem{Vil} C. Villani, {\em Topics in optimal transportation}, AMS, 2003, vol. 58.

\bibitem{Wak} A. Wakolbinger, Schr\"{o}dinger Bridges from 1931 to 1991,
in: E. Caba$\tilde{n}$a et al. (eds) , {\em Proc. of the 4th Latin
American Congress in Probability and Mathematical Statistics}, Mexico
City 1990, Contribuciones en probabilidad y estadistica matematica 3
(1992) , 61-79.
\end{thebibliography}
\end{document}